\documentclass[12pt]{amsart}
\usepackage{amsthm}
\usepackage{fullpage}
\usepackage[onehalfspacing]{setspace}
\usepackage{amsmath}
\usepackage{amsfonts}
\usepackage{hyperref}
\usepackage{enumitem}
\usepackage[maxbibnames=99, backend=biber]{biblatex}
\usepackage{amssymb}
\usepackage{amsfonts}
\usepackage{braket}

\newtheorem{proposition}{Proposition}
\newtheorem{theorem}{Theorem}
\newtheorem{claim}{Claim}

\newtheorem{lemma}{Lemma}
\newtheorem{conjecture}{Conjecture}

\newtheorem{condition*}{Definition}

\newcommand{\Tr}{\operatorname{Tr}}
\newcommand{\lsim}{\lesssim}
\newcommand{\gsim}{\gtrsim}

\usepackage{xcolor}
\definecolor{MidnightBlue}{RGB}{25,25,150}
\definecolor{BrickRed}{RGB}{182,50,28}
\definecolor{ForestGreen}{RGB}{34,139,34}

\title{Localization of One-Dimensional Random Band Matrices}
\author{Reuben Drogin}

\begin{document}

\begin{abstract}
We consider a general class of $n\times n$ random band matrices
with bandwidth $W.$ When $W^2\ll n$, 
we prove that with high probability
the eigenvectors of such matrices are localized and 
decay exponentially at the sharp scale $W^2$.   
Combined with the delocalization results of
Yau and Yin \cite{Ya-Yin}, and Erd\H{o}s and Riabov \cite{erdos-riabov}, 
this establishes the conjectured 
localization-delocalization transition for a large class of 
random band matrices. 
\end{abstract}
\maketitle
\section{Introduction}
\subsection{Model and Results}
Consider a symmetric random matrix $H\in (\mathbb{R}^{W\times W})^{N\times N}$ 
in the following block tri-diagonal form: 
\begin{equation}\label{eq:definition-of-H}
    H :=  
          \begin{pmatrix}
            A_{1} & B_{1} & \dots & 0\\
            B_{1}^* & A_{2} & \ddots & \vdots \\
            \vdots & \ddots & \ddots &B_{N-1}\\
            0 & \dots & B_{N-1}^{*}& A_{N} 
            \end{pmatrix},
\end{equation} 
where the $A_i$ and $B_i$ are $W\times W$ matrices and $A_i = A_i^{\ast}$. 
We assume the entries of 
$H$ are independent up to the symmetry constraint $H = H^{\ast}$, 
and the entries of $\left(\sqrt{W}A_i\right)_{i\in \left[1,N\right]}$ 
and $\left(\sqrt{W}B_i\right)_{i\in \left[1,N-1\right]}$ are \textit{$M$-regular}, for some $M>0$. 

\begin{condition*}[M-Regular]\label{assumption:general-distribution}
For any $M>0$, an $\mathbb{R}$-valued random variable $X$ is $M$-regular if 
$\mathbb{E}X = 0$, $\mathbb{E}X^2 \leq 1$, 
$\mathbb{E}X^4 \leq M,$ and it has a $C^2$ 
density $\phi:\mathbb{R}\to \mathbb{R}_{>0}$, satisfying
\begin{equation}\label{eq:general-distribution-condition}
\left\| \frac{d^2}{dx^2} \left(\log\phi\right)\left(x\right)\right\|_{L^\infty(\mathbb{R})} \leq M.
\end{equation}
\end{condition*}

Examples of $M$-regular distributions 
include the unit Gaussian but also more general random variables 
such as those with a density proportional
to $\frac{1}{1+|x|^{\alpha}}$ for some $\alpha>5$. 
In particular, $H$ may be drawn from the real Wegner orbital model, 
in which the $A_i$ and $B_i$ are properly normalized independent GOE and Ginibre matrices, respectively.

If we let $n=NW$, then $H$ is an $n\times n$ \textit{random band matrix} (RBM) 
with bandwidth of order $W$. It is conjectured the eigenvectors of such matrices 
are typically \textit{exponentially localized} when 
$W^2/n\ll 1$ and \textit{delocalized} when $W^2/n\gg 1$. See the 
discussion in Section 1.2. Our main result rigorously proves 
the localization part of 
this conjecture for the class of RBMs above.




\begin{theorem}\label{thm:general-main}
Let $M,E_0>0$ and $H$ be as in \eqref{eq:definition-of-H}
with independent entries up to the constraint $H=H^{\ast}$.   
If the nonzero entries of $\sqrt{W}H$ are $M$-regular, 
and $\left(\psi_j, E_j\right)_{j\in \left[1,NW\right]}$ are
the normalized eigenvector-eigenvalue pairs of $H$, then for any $N,W>0$
and $x,y\in [1,NW]$ we have
\begin{equation}\label{eq:conclusion-of-thm-1}
\mathbb{E}\left[\sum_{j:|E_j|\leq E_0}\left|\psi_{j}\left(x\right)\right|\left|\psi_j\left(y\right)\right|\right]\leq 2W^Ce^{-\frac{c|x-y|}{W^2}}.
\end{equation}
The constants $C,c>0$ depend only on $M$ and $E_0$. 
\end{theorem}

Combined with the delocalization results in  \cite{Ya-Yin} and \cite{erdos-riabov}, 
Theorem \ref{thm:general-main} establishes the 
conjectured localization-delocalization transition for 
a large class of one-dimensional RBMs. 
Specifically, let $H$ be as in Theorem \ref{thm:general-main}
and additionally suppose the variances in each row sum to $1$, 
and the entries of $\sqrt{W}H$ have uniformly bounded moments of all orders. 
If $W^2/n\gg 1$, Corollary 3.5 of \cite{erdos-riabov} implies 
that all eigenvectors 
of $H$ with eigenvalues in $[-2+\epsilon, 2-\epsilon]$
are delocalized with high probability.
Conversely, if $W^2/n\ll 1$, Theorem \ref{thm:general-main} implies 
that all such 
eigenvectors of $H$ are exponentially localized with high probability.  

Briefly, we comment on the condition $|E_j|\leq E_0$ 
in \eqref{eq:conclusion-of-thm-1}. 
If $W\geq n^{\epsilon}$ and the entries 
of $\sqrt{W}H$ have sufficiently many uniformly bounded moments, depending on $\epsilon$, 
then it is well known that $\sigma(H)\subset [-E_0,E_0]$ with high probability, for $E_0$ large. 
Alternatively, the condition $|E_j|\leq E_0$
can be removed completely via well-known Lifschitz tailing arguments; 
see Remark 2.2 in \cite{Schenker-Shapiro-etc-2022}.

By standard arguments, 
see Theorem A.1 of \cite{Aizenman-frac-moment} or Chapter 7 of \cite{aizenman-warzel}, 
Theorem \ref{thm:general-main} follows from 
Theorem \ref{thm:single-energy-decay} below. 
To state the  theorem 
define, for any $E\in \mathbb{R}$, the operator $G\left(E\right)$ 
by 
$$G\left(E\right):= \left(H-E\right)^{-1}.$$
It is well known that for any fixed $E\in \mathbb{R}$, 
$G(E)$ exists almost surely under the assumptions on $H$ above 
(see Proposition \ref{prop:wegner}). We often leave the $E$ dependence 
implicit and write $G\left(i,j\right)$ for the $W\times W$ matrix 
which is the $\left(i,j\right)$-th block of $G\left(E\right)$.  

\begin{theorem}\label{thm:single-energy-decay}
Let $M,E_0>0$ and $H$ be as in \eqref{eq:definition-of-H} 
with independent entries up to the constraint $H = H^{\ast}$.
If all non-zero entries of $\sqrt{W}H$ are $M$-regular,
then for any $q\in (0,1/5],\ 
|E|\leq E_0$, $N,W>0$, and unit vector $w\in \mathbb{R}^W$, we have 
\begin{equation}\label{eq:conclusion-of-thm-2}
\left(\mathbb{E}\left\|G\left(1,N\right)w\right\|^{q}\right)^{1/q}\leq CW^{C}e^{-\frac{cN}{W}}.
\end{equation}
The constants $C,c>0$ depend only on $M$ and $E_0$.
\end{theorem}

Note that, due to the block structure, the entries in $G\left(1,N\right)$
correspond to pairs of indices separated by the order $NW$.

\subsection{Background and Motivation}
RBMs are matrices with random entries 
that vanish or decay outside of 
a band around the diagonal. They are of interest 
because they model 
various phenomena arising in quantum chaos and the Anderson model. 
Some of the earliest 
numerical results on RBMs arose in 
connection to the quantum kicked rotator, 
which is the quantum analogue of the classical Chirikov standard map. 
The associated dynamics on $\ell^2(\mathbb{R}/2\pi\mathbb{Z})$ are generated by the time-periodic Hamiltonian 
$$H(\theta,t) = \frac{\partial^2}{\partial \theta^2}- k\cos(\theta)\delta_\tau(t)$$
for some $k,\tau>0$, where $\delta_\tau$ is the $\tau$-periodic delta function. 
In Fourier space, the time-$\tau$
mapping corresponds to an infinite unitary matrix with `pseudo-random' entries 
decaying exponentially outside a band of width $k$. 
The papers \cite{kicked-rotator,Chirikov,Kicked-rotor-3} found 
that for typical choices of $\tau$, the 
momentum distribution of a wavefunction spreads diffusively
for time on the order $k^2$, beyond which it stalls, and 
argued similar behavior when the entries of the matrix are 
fully random. 
This suggested the eigenvectors of these banded matrices, 
are localized on the scale $k^2$. For a broader discussion of quantum chaos and 
RBMs, see the survey \cite{Casati-Chirikov-1995}.

In the early 1990s, motivated by the connections mentioned above, 
many authors \cite{Casati-et-al-1990,Casati-et-al-1990-2,Seligman-et-al},
numerically studied the spectral behavior of $n\times n$ RBMs, 
as the bandwidth, given by $W$, varied. They found the ratio $W^2/n$ governs the spectral 
behavior and conjectured the following: 
\begin{conjecture}
\label{conj:delocalization-localization}
Let $H\in \mathbb{R}^{n\times n}$ be a random band matrix with 
bandwidth $W$. 
\begin{enumerate}
\item If $W^2/n\ll 1$, the eigenvectors of $H$
are exponentially localized to the scale $W^2$ and the local eigenvalue 
process rescales to a Poisson point process.
\item If $W^2/n\gg 1$, the eigenvectors of $H$ are delocalized 
and the local eigenvalue process resembles that of a 
GOE matrix. 
\end{enumerate} 
\end{conjecture}

The first theoretical support for Conjecture \ref{conj:delocalization-localization} came from 
Fyodorov and Mirlin
\cite{Fyodorov-et-al-1991,Fyodorov-et-al-1993}, who used a non-rigorous 
supersymmetric approach to analyze a specific
Gaussian RBM decaying rapidly outside the band of width $W$. 
They showed a localization-delocalization transition occurred at $W^2/n\approx 1$, 
along with changes in the local eigenvalue statistics. Further support 
came from scaling arguments \cite{Thouless}, and arguments 
based on transfer matrices. We refer the reader 
to the survey \cite{physics-survey} and the references 
therein for more on the physics background. 

Recently, RBMs have received increased interest due to their connection 
to the Anderson model. Introduced in Anderson's seminal paper \cite{Anderson} as a model for 
electron transport,
the Anderson Hamiltonian $H$ on $\ell^2\left(\mathbb{Z}^d\right)$ is given by
$$\left(Hf\right)(x) = \left(\Delta_{\mathbb{Z}^d}f\right)(x) + \lambda V_xf(x) ,$$
where $\lambda>0$ is a coupling constant, $\Delta_{\mathbb{Z}^d}$ is the discrete Laplacian on
$\mathbb{Z}^d$, 
and $\left(V_x\right)_{x\in \mathbb{Z}^d}$ are i.i.d. random variables. 
Anderson argued $H$ exhibits a phase transition:
when $\lambda$ is large, $H$ is almost surely \textit{localized}, meaning 
it has an orthonormal basis in $\ell^2(\mathbb{Z}^d)$ of 
exponentially decaying eigenfunctions and when 
$\lambda$ is small and $d\geq 3$, $H$ 
is almost surely \textit{delocalized} in the sense that it has an interval 
of absolutely continuous spectrum, and non-$\ell^2$ eigenfunctions. 
This localization phenomenon, known as \textit{Anderson localization}, 
is well understood mathematically. It has been proven
in $d\geq 1$ for $\lambda$ large \cite{Frolich-Spencer,Aizenman-1994}, 
and in $d=1$ for any $\lambda>0$ \cite{Kunz-Souillard}. 
On the other hand, delocalization 
in $d\geq 3$ remains a major open problem, known as the \textit{extended states conjecture}.
RBMs exhibit both localization and delocalization, but interpolate
between the Anderson model, which is random only on the diagonal, 
and better understood mean-field models where most or all entries are random. 
We refer the reader to the surveys \cite{bourgade2018,Spencer} for more on the mathematical background of Conjecture \ref{conj:delocalization-localization}
and to the papers
\cite{dubova-yang-yau-yin,erdos-riabov,BDH,yang-yin,Schenker-Shapiro-etc-2022} 
for more on RBMs and their connection to the Anderson model.

% ...  


\subsection{Prior Mathematical Work}

The first rigorous mathematical progress 
on the localization side of Conjecture \ref{conj:delocalization-localization} 
was due to Schenker \cite{Schenker-2009}. 
He proved localization in a large class of $n\times n$ Gaussian RBMs 
when $W^8/n\ll 1$, and for a
wider class of non-Gaussian models he proved 
there exists a $C>0$ such that 
localization holds when $W^C/n\ll 1$. The argument, explained below, 
combines a fluctuation estimate for $G(1,N)$ with an a-priori estimate, and 
was first proposed by Michael Aizenman. 

Subsequently, Schenker's argument was refined by 
several groups of authors. 
Using a new sharp Wegner estimate for Gaussian matrices, 
Peled, Schenker, Shamis, and Sodin \cite{Peled-Schenker-Shamis-Sodin}
proved localization in a class of Gaussian RBMs
when $W^7/n\ll 1$. Later, 
this was improved to $W^4/n\ll 1$ in the real Wegner orbital model,
which is obtained by taking $A_i$ and $B_i$ to be GOE and Ginibre matrices, 
respectively. This was done contemporaneously by
Chen and Smart \cite{Charlie-Nixia},
and Cipolloni, Peled, Schenker, and Shapiro \cite{Schenker-Shapiro-etc-2022}
(whose proof extends to mixtures of Gaussian models). Most recently, a preprint 
of Goldstein \cite{Goldstein} claims to prove
localization in the real Wegner orbital model when $N\geq W^2$, i.e. $W^3/n\ll 1$,
with eigenfunctions localized at the essentially sharp scale $(\log W)^3W^2$. 
It is possible that the ideas in \cite{Goldstein} 
may be sharpened to reach the regime $W^2/n\ll 1$ in that model; 
however, to our knowledge, the community has been unable to verify the proof, 
and shortly after posting \cite{Goldstein}, Goldstein passed away.

We also mention results using alternative approaches. 
In \cite{scherbina-1,scherbina-2,scherbina-3},
M. Shcherbina and T. Shcherbina 
rigorously implement supersymmetric methods to analyze  
a Gaussian RBM with a specific variance profile. 
In particular, they prove a transition at $W^2\approx n$ for the moments 
of characteristic polynomials and 
two-point functions. Finally, we mention the interesting work 
of Shapiro \cite{shapiro}, which is notable for its use of Lyapunov exponents. 
Using the explicit form of the Lyapunov exponents 
of Ginibre matrices, he proves 
a version of \eqref{eq:conclusion-of-thm-2} 
for a special Gaussian RBM when $E=0$ and $N$ is 
sufficiently large depending on $W$.

We emphasize that, beyond providing the first complete proof 
of the localization part of
Conjecture \ref{conj:delocalization-localization},
Theorem \ref{thm:general-main}
is one of the only localization results 
which does not rely on explicit computations with Gaussian densities 
or on symmetries and invariances.

The delocalization side of Conjecture \ref{conj:delocalization-localization}
was recently resolved. 
In the classical Wegner orbital model, i.e. $A_i$ and $B_i$ are GUE and complex Ginibre matrices, 
Yau and Yin \cite{Ya-Yin} proved delocalization
of eigenvectors, GUE-type eigenvalue statistics, and QUE 
when $W^2/n\gg 1$. 
The work of Erd\H{o}s and Riabov \cite{erdos-riabov}
extended these results to a very general class of band matrices. 
These two works are part of a recent series of spectacular breakthroughs 
on delocalization in RBMs. For instance see
\cite{dubova-yang-yau-yin,Xu-Y-cubed,2d-band-matrices,yang-yin}.

\subsection{Proof Sketch}
To begin, we briefly recall the outline of the Schenker method, 
introduced in \cite{Schenker-2009}.
First, we fix a unit vector $w\in \mathbb{R}^W$ and decompose $\log \left\|G\left(1,N\right)w\right\|$ 
into a sum of random variables. 
To do this we iterate the resolvent formula to obtain
\begin{equation}
\begin{aligned}\label{eq:decomposition-in-outline}
G\left(1,N\right) 
& = (-1)^{N-1}G_{\left[1,1\right]}\left(1,1\right)B_1G_{\left[1,2\right]}\left(2,2\right)\dots B_{N-1}G_{\left[1,N\right]}\left(N,N\right)\\
& :=  (-1)^{N-1}D_1^{-1}B_1D_2^{-1}...B_{N-1}D_N^{-1},
\end{aligned}
\end{equation}
where $D_{j}^{-1}:= G_{\left[1,j\right]}\left(j,j\right)$ is the $\left(j,j\right)$ block of $\left(H_{\left[1,j\right]}-E\right)^{-1}$.
Then we can write 
$$\log \left\|G\left(1,N\right)w\right\| = \sum_{j=1}^{N} \alpha_j$$
where $\alpha_j$ is $\log$ of the contribution from $D_j^{-1}B_j$ given by 
\begin{equation}\label{def:first-def}
\alpha_j := \log \left\|D_j^{-1}B_j v_j\right\|, \text{ where } v_j:=\frac{D_{j+1}^{-1}B_{j+1}...D_N^{-1}w}{\left\|D_{j+1}^{-1}B_{j+1}...D_N^{-1}w\right\|}
\end{equation}
for $j\in \left[1,N-1\right]$ and $\alpha_N := \log \left\|D_N^{-1}w\right\|.$

Second, one argues it essentially suffices to prove \textit{each $\alpha_j$ fluctuates 
at the scale $\frac{1}{\sqrt{W}}$}. To see this, suppose 
the $\alpha_j$ were completely independent so that 
for any $q\in \mathbb{R}$
\begin{equation}\label{eq:product-form-0}
\mathbb{E}\left\|G(1,N)w\right\|^q  = \prod_{j=1}^{N}\mathbb{E}e^{q\alpha_j}.
\end{equation}
By convexity, a random variable $X\in \mathbb{R}$ 
\textit{fluctuating at the scale $\epsilon$}, in the sense that 
\begin{equation}\label{eq:non-concentration}
\sup_{a\in \mathbb{R}}\mathbb{P}\left( \left|X-a\right|\leq \epsilon\right)\leq 1-c,
\end{equation}
satisfies
$$\left(\mathbb{E}e^{X}\right)^2 \leq e^{-c\epsilon^2}\left(\mathbb{E}e^{2X}\right),$$
if $\epsilon\in \left[0,1\right]$. See Lemma \ref{lem:elementary-jensens}. Hence if one proves 
\eqref{eq:non-concentration} with $\epsilon = \frac{1}{\sqrt{W}}$ for each $\alpha_j$, 
and the $\alpha_j$ were independent, \eqref{eq:product-form-0} 
would imply
$$
\left(\mathbb{E}\left\|G(1,N)w\right\|^{q}\right)^2
\leq   \prod_{j=1}^{N}e^{-cq^2W^{-1}}\mathbb{E}e^{2q\alpha_j} 
= e^{-cq^2\frac{N}{W}} \mathbb{E}\left\|G(1,N)w\right\|^{2q}
 \leq W^{C}e^{-cq^2\frac{N}{W}},
$$
as long as $0\leq 2q<1$ so that we can use the Wegner estimate (Proposition \ref{prop:wegner}) 
in the last inequality. In practice, the $\alpha_j$ 
are \textit{only independent conditioned on $\left(D_j, v_j\right)_{j\in \left[1,N\right]}$} 
(see Proposition~\ref{prop:schenker-product}), and so we need a version of \eqref{eq:non-concentration} which is 
\textit{conditional on realizations of $\left(D_j, v_j\right)_{j\in \left[1,N\right]}$}.

The main new contribution of this paper 
is a conditional version of \eqref{eq:non-concentration}. 
If the $A_i$ are GOE matrices, 
and we define the $\sigma$-algebra $\mathcal{F} = \sigma\left(\left(D_j, v_j\right)_{j\in \left[1,N\right]}\right)$, 
we prove
\begin{equation}\label{eq:single-step-variance-GOE}
\sup_{a\in \mathbb{R}}\mathbb{P}\left(\left|\alpha_j-a\right|\leq \frac{1}{\sqrt{W}} \,\middle| \, \mathcal{F}\right)\leq 1-c \mathbb{P}\left(\left\|B_jv_j\right\|, \left\|A_{j+1}v_{j}\right\|, \left\|B_j^{\ast} D_j^{-1}B_jv_j\right\|\lsim 1 \,\middle| \, \mathcal{F}\right).
\end{equation}
See Lemma \ref{lem:single-step-variance-general} for the general statement. 
The norms on the RHS are typically $\lsim 1$ and so Theorem \ref{thm:single-energy-decay} follows 
quickly by adapting the argument above. We note that proving $\left\|B_j^{\ast}D_j^{-1}B_jv_j\right\|\lsim 1$ 
is subtle and cannot come from a naive operator norm estimate because $\left\|D_j^{-1}\right\|$ is typically of size $W$.  See 
Claim \ref{claim:either-or} and the discussion above it. 

The idea of the proof of  \eqref{eq:single-step-variance-GOE} is as follows.
Once we have conditioned on $\left(D_j, v_j\right)_{j\in \left[1,N\right]}$, 
the only randomness left in $\alpha_j$ comes from $B_j$. Hence 
the problem reduces to proving that conditioned on 
$\mathcal{F}$, the random variable 
$\log \|B_jv_j\|,$ 
fluctuates at the scale $\frac{1}{\sqrt{W}}$. 
Note that conditioning on $\mathcal{F}$ changes 
the law of the $B_j$, and fixes the direction of $B_jv_j$. 
To lower bound the fluctuations we consider the effect of replacing $B_j$ via
\begin{equation}\label{eq:perturbation}
B_j \mapsto B_{j} \pm \frac{2}{\sqrt{W}} \ket{B_jv_j}\bra{v_j}, 
\end{equation} 
where $v_j$ is as in \eqref{def:first-def}. Both choices of sign in \eqref{eq:perturbation} vary $\alpha_j$ by $2W^{-1/2}$
and maintain the direction of $B_jv_j$. 
Hence, if we show at least one choice of sign in \eqref{eq:perturbation} 
decreases the density of $(B_j\ | \mathcal{F})$ by at most an $O(1)$ factor, we 
expect $\alpha_j$ to fluctuate at the scale $W^{-1/2}$.
This argument is abstractly formulated in Lemma \ref{lem:MW}. 
The change in the density of $\left(B_j|\mathcal{F}\right)$ 
under \eqref{eq:perturbation} is estimated by Taylor expansion, 
and the quantities $\left\|B_jv_j\right\|, \left\|A_{j+1}v_j\right\|,$ and $\left\|B_j^{\ast}D_j^{-1}B_jv_j\right\|$ 
arise naturally from this computation.  
See Lemma \ref{lem:distortion-general} for the details. 
The proof of Lemma \ref{lem:MW}
is related to Pfister's \cite{Pfister} and Dobrushin and Shlosman's \cite{Dobrushin-Shlosman} proof of the Mermin-Wagner 
Theorem in Statistical Mechanics, and we note this idea was
first introduced to RBMs in \cite{Schenker-Shapiro-etc-2022}.

We remark that if one writes $B_j$ in a basis with $v_j$ as the first basis vector, 
\eqref{eq:perturbation} simply multiplies the first 
column of $B_j$ by $1\pm 2W^{-1/2}$.  
Past arguments generated variation in $\alpha_j$ by essentially multiplying \textit{the entire matrix $B_j$}, 
by a factor $1\pm \delta$. This is a more expensive perturbation of $B_j$ 
in the sense that it creates a large change the density of $B_j$ 
but a relatively small fluctuation of $\left\|B_jv_j\right\|$.

\subsection{Further Directions and Open Questions}
We list some open questions related to this work.

\begin{enumerate}
\item \textbf{Other distributions of entries:}
For bounded reasonably smooth distributions, the needed Wegner estimate still holds, 
and it is likely one 
could prove a version of Lemma \ref{lem:single-step-variance-general}
via a similar approach.
On the other hand, for more singular distributions, like Bernoulli, 
the proof seems very difficult to adapt.
\item \textbf{Other band matrix models.}
The localization length of $W^2$
for eigenvectors  
is expected to be universal, in the sense 
that a version of Theorem \ref{thm:general-main} 
should hold irrespective of the specific form of the matrix. 
It would be ideal 
to have a result which holds 
for an even more general class of matrices. 
Other models of interest include 
the Block Anderson model in which the $B_i$ are 
identity matrices, or the 
proper RBM model in which the $B_i$ are 
lower triangular matrices. We hope to address these models 
in future work. 
\item \textbf{Poisson Eigenvalue Statistics.}
It is conjectured that in the localization regime, 
the eigenvalue process of
$H$, properly rescaled, converges to a Poisson process. 
To prove 
this one need to pair the localization result proven here, 
with sufficient control on the density of states. 
The latter is the main obstacle. See Section 1.2
of \cite{Schenker-2009} for a discussion. 
\item \textbf{Lyapunov Exponents.} 
Decay properties of the eigenfunctions of $H$ are connected to 
the study of products of the random 
$2W\times 2W$ matrices 
$$
T_i:=
\begin{pmatrix}
A_i -E & -B_i^{-1}\\
B_i & 0
\end{pmatrix}. 
$$
Indeed, Theorem \ref{thm:general-main}
is roughly equivalent to showing the
positive Lyapunov exponents of $T_1$
are larger than $\frac{c}{W}$. While 
Furstenberg's Theorem implies the qualitative statement 
that the Lyapunov exponents of $T_1$ are nonzero, 
quantitative statements seem difficult to prove. See the paper 
of Shapiro \cite{shapiro}, however, for a special case where 
this can de done.
For more on Lyapunov exponents and localization see 
the book \cite{carmona-lacroix}.

\item \textbf{Anderson Model on the Strip.}
It is conjectured that the Anderson model 
on the strip $\mathbb{Z}\times \left[1,W\right]$
has a localization length which is polynomial in $W$. 
The current best upper bound is $e^{CW\log(W)}$ 
due to Bourgain \cite{Bourgain-2013}. See 
\cite{Binder-Goldstein-Voda,Binder-Goldstein-Voda-2} for related results on 
this model.

\end{enumerate}

\subsection{Notation and Conventions}
We use $C,c>0$ as constants that
may vary in each appearance, 
and depend on $E_0, M$, but are always uniform in $N,W$. 
We say $X\lsim Y$ if $X\leq CY$. When writing probability 
densities we use $Z$ as a normalization constant.

We will often omit $E$ when writing $G(E)$, 
for instance writing $G(i,j)$ for the $W\times W$ matrix 
which is the $(i,j)$-th block of $G(E)$. Furthermore, 
for any $S\subset \left[1,N\right]$, we let 
$H_{S}$ be the restriction of $H$ to the rows and columns indexed by $S$, 
and $G_S$ be the corresponding resolvent, given by
$$G_{S}(E):= \left(H_S-E\right)^{-1}.$$

\subsection{Acknowledgements}
We thank Charles Smart for helpful discussions, and Adam Black 
and Felipe Hern\'andez for helpful comments. 
\section{Schenker Decomposition and Wegner Estimate}\label{sec:preliminaries}
In this section we recall two standard 
tools. First is a decomposition of $G\left(1,N\right)$
due to Schenker \cite{Schenker-2009}.  
It writes $G\left(1,N\right)$ as a product of 
matrices having some independence. 

\begin{proposition}(Schenker Decomposition)\label{prop:schenker-product}
Fix $E\in \mathbb{R}$ and define $D_1 := A_1-E$, and
$$D_{j+1} := A_{j+1} - E - B_j^{\ast}D_j^{-1}B_j,$$
for $j\in \left[1,N-1\right]$. Then 
\begin{equation}\label{eq:schenker-product}
G_{\left[1,N\right]}\left(1,N\right) = (-1)^{N-1}D_1^{-1}B_1D_2^{-1}...B_{N-1}D_N^{-1},
\end{equation}
and for any $j\in \left[1,N-1\right]$, $B_j$ is independent of $(B_i)_{i\neq j}$ conditioned on $D_{j}, D_{j+1}$.
\end{proposition}
A proof can be found in Section 4 of \cite{Schenker-2009}.
We briefly recall it here for the readers convenience. 

\begin{proof}
Expanding $\left(H-E\right)^{-1}$ about $\left(H_{\left[1,N-1\right]}-E\right)^{-1}$ 
via the resolvent formula gives 
$$G_{\left[1,N\right]}\left(1,N\right) = -G_{\left[1,N-1\right]}(1,N-1)B_{N-1}G_{\left[1,N\right]}\left(N,N\right).$$
Iterating gives
\begin{equation}\label{eq:schenker-product-step-1}
    G_{\left[1,N\right]}\left(1,N\right) = (-1)^{N-1}G_{\left[1,1\right]}\left(1,1\right)B_1G_{\left[1,2\right]}\left(2,2\right)...B_{N-1}G_{\left[1,N\right]}\left(N,N\right).
\end{equation}
To show \eqref{eq:schenker-product} 
it remains to prove $D_j^{-1} = G_{[1,j]}(j,j)$. 
For this, note that the Schur complement formula (Proposition \ref{prop:schur-complement} in the appendix) gives
$$G_{\left[1,j\right]}\left(j,j\right) = \left(A_j - E - B_j^{\ast}G_{[1,j-1]}(j-1,j-1)B_j\right)^{-1}$$
for each $j\in \left[2,N\right]$. Since $G_{\left[1,1\right]}\left(1,1\right) = \left(A_1-E\right)^{-1}$, relation \eqref{eq:schenker-product} follows. 

The last observation can be seen from the joint density 
of $(D_i)_{i\in \left[1,N\right]}$ and $(B_i)_{i\in \left[1,N-1\right]}$. If $\phi_{1,i}$ and $\phi_{2,i}$ are
the densities of $A_i$ and $B_i$, then the density of the random variable $\left(\left(D_i\right)_{i\in \left[1,N\right]}, \left(B_i\right)_{i\in \left[1,N-1\right]}\right)$
is given by 
\begin{equation}\label{eq:joint-density}
\phi_{1,1}\left(D_1+E\right)\prod_{j=1}^{N-1}\phi_{1,{j+1}}\left(D_{j+1}+E+B_j^{\ast}D_j^{-1}B_j\right)\prod_{j=1}^{N-1}\phi_{2,j}\left(B_j\right).
\end{equation}
Here we used that the $A_i$ and $B_i$ are all independent and changed variables 
with the map $\Phi:\left(A_1,...,A_N,B_1,...,B_{N-1}\right)\mapsto \left(D_1,...,D_N,B_1,...,B_{N-1}\right),$
noting that it has jacobian 1. 
In \eqref{eq:joint-density}, $B_j$ only appears in the product in terms with $D_j$ and $D_{j+1}$, proving the proposition.  
\end{proof}

We will also use the following standard Wegner estimate, 
see Theorem 7 in \cite{Schenker-2009} or Lemma 1 in \cite{Bourgain-2013}. 
For completeness we include a proof in the appendix. 
\begin{proposition}(Wegner Estimate)\label{prop:wegner}
Let $M>0$ and $H$ be as in Theorem \eqref{thm:general-main}.
Then, for any $E\in \mathbb{R}$, $\lambda>0$,  and $i,j\in \left[1,N\right]$, we have
$$\mathbb{P}\left(\left\|G\left(i,j\right)\right\|>\lambda\right) \lsim W^{3/2}\lambda^{-1}.$$
\end{proposition}

\section{Fluctuations Lower Bound}

In this section we prove the key lemma. 
Applying Proposition \ref{prop:schenker-product}
to $G(1,N)w$, we define the vectors $v_N := w$ and
\begin{equation}\label{def:v_j}
v_{j}:= \frac{D_{j+1}^{-1}B_{j+1}...D_N^{-1}w}{\left\|D_{j+1}^{-1}B_{j+1}...D_N^{-1}w\right\|},
\end{equation}
for all $j\in \left[1,N-1\right]$. Similarly, we define the quantities $\alpha_N := \log \left\|D_N^{-1}v_N\right\|$ 
and 
\begin{equation}\label{def:alpha_j}
  \alpha_j:= \log \left\|D_j^{-1}B_jv_j\right\| 
\end{equation}
for all $j\in \left[1,N-1\right]$. By Proposition \ref{prop:schenker-product}, we have 
that $\log \left\|G\left(1,N\right)w\right\| = \sum_{i=1}^N\alpha_i,$
and that conditioned on the $\sigma$-algebra 
\begin{equation}\label{def:sigma-algebra}
  \mathcal{F} := \sigma\left(\left(D_i,v_i\right)_{i\in \left[1,N\right]}\right),
\end{equation}
the random variables $(\alpha_i)_{i\in \left[1,N\right]}$ 
are independent.
Finally, define the map $f_j:\mathbb{R}^{W\times W}_{Sym}\to \mathbb{R}^{W\times W}_{Sym}$ by
\begin{equation}\label{def:f-function}
f_{j}\left(A\right) := \frac{1}{W}\nabla\left(\log\phi_{1,j}\right)(A),
\end{equation}
where the gradient $\nabla$ is with respect inner product $\langle A, B\rangle = Tr(AB)$ on $\mathbb{R}^{W\times W}_{Sym}$. 
Note that if $A_j$ is a normalized 
GOE matrix, then $\phi_j(A) = e^{-\frac{W}{4}\text{Tr}A^2}$, and so 
$f_j(A) = \frac{1}{2}A.$ We encourage the reader to keep the GOE case in
mind upon first reading. The following says that \textit{conditioned on}
$\left(D_i, v_i\right)_{i\in \left[1,N\right]}$, 
$\alpha_j$ fluctuates at scale 
scale $W^{-1/2}$, as long as the quantities on the RHS below are typically $\lsim 1$.

\begin{lemma}\label{lem:single-step-variance-general}
Let $M>0$ and $H$ be as in Theorem~\ref{thm:general-main}. 
Then, for any $j\in \left[1,N-1\right]$, $E\in \mathbb{R}$, $t\geq 1$, and unit vector $w\in \mathbb{R}^{W}$,
$$\sup_{a\in \mathbb{R}}\mathbb{P}\left(\left|\alpha_j - a\right| \leq \frac{1}{\sqrt{W}}  \,\middle| \, \mathcal{F}\right)
\leq 1 - e^{-Ct}\mathbb{P}\left(\left\|B_jv_j\right\|, \left\|f_{j+1}(A_{j+1})v_j\right\|, \left\|B_j^{\ast}D_j^{-1}B_jv_j\right\| \leq t \,\middle| \, \mathcal{F}\right).$$
Here $v_j$, $\alpha_j$, $\mathcal{F}$ and $f_j$ are as given in \eqref{def:v_j}, \eqref{def:alpha_j}, \eqref{def:sigma-algebra} and \eqref{def:f-function} respectively. 
\end{lemma}

The idea is to replace $B_j$ with $B_j \pm \delta \ket{B_jv_j}\bra{v_j}$ 
and compute the distortion of the density of $B_j$. 
Either choice of sign changes $\alpha_j$ 
by $\delta$, and so if one of the choices usually decreases the density of 
$B_j$ by at most an $O(1)$ factor, then $\alpha_j$ should fluctuate at the scale $\delta$.
This argument is made rigorous in the following Mermin--Wagner type lemma. 
The work \cite{Schenker-Shapiro-etc-2022} was the first to lower bound 
 fluctuations in this setting using a Mermin-Wagner argument. 
On first 
reading it is instructive to imagine $X$ is a Gaussian on $\mathbb{R}$,
$T_{\pm}(x) = \left(1\pm \delta\right)x$, and $g(x) = \log \left|x\right|$. 

\begin{lemma}\label{lem:MW}
  Let $X\in \mathbb{R}^d$ be a random variable 
  with a continuous density $\Phi: \mathbb{R}^d\to \mathbb{R}_{>0}$
  and 
  $g:\mathbb{R}^d\to \mathbb{R}$ be a measurable function. If there exists 
  $\delta>0$ and  smooth diffeomorphisms
  $T_+,T_-:\mathbb{R}^d\to \mathbb{R}^d$ 
  such that 
  \begin{align*}
  g(T_{\pm}(x))& = g(x) \pm \delta,
  \end{align*}
  for almost every $x\in \mathbb{R}^d$, then for any $t\in \left[0,1\right]$
  $$\sup_{a\in \mathbb{R}}\mathbb{P}\left(\left|g(X)-a\right|\leq\frac{\delta}{2}\right) \leq 1- c\sqrt{t}\mathbb{P}\left(\frac{T_+^{\#}\Phi}{\Phi}(X) \frac{T_-^{\#}\Phi}{\Phi}(X)\geq t\right).$$
  \end{lemma}
  Recall that $T^{\#}_{\pm}$ is the pullback map associated to $T_{\pm}$, and is given by 
  \[T^{\#}_{\pm}\Phi\left(x\right) = \Phi\left(T_{\pm}(x)\right)\text{Jac}\left(T_{\pm}\right)\left(x\right).\]   
\begin{proof}
Define 
$$p := \mathbb{P}\left(\frac{T_+^{\#}\Phi}{\Phi}(X) \frac{T_-^{\#}\Phi}{\Phi}(X)\geq t\right),$$
and suppose for contradiction, that for some $a_0\in \mathbb{R}$, 
\begin{equation}\label{eq:contradiction-assumption}
  \mathbb{P}\left(\left|g(X)-a_0\right|\leq \frac{\delta}{2}\right) \geq 1-\frac{1}{8}\sqrt{t}p.
\end{equation}
Since $t\in [0,1]$, the law of total probability implies that, for either $\sigma = +$ or $\sigma = -$, 
$$\mathbb{P}\left(|g(X)-a_0|\leq \frac{\delta}{2} \text{ and }\frac{T_{\sigma}^{\#}\Phi}{\Phi}(X)\geq \sqrt{t} \right) \geq \frac{1}{4}p.$$
But then we can change variables by $T_{\sigma}$ to estimate
\begin{align*}
\mathbb{P}\left(\left|g(X)-a_0\right|\geq \frac{\delta}{2} \right)
& = \int_{\mathbb{R}} 1_{\left|g(x)-a_0\right|\geq \frac{\delta}{2}} \Phi(x) dx\\
& \geq \int_{\mathbb{R}} 1_{\left|g(x)-a_0\right|\leq \frac{\delta}{2}} \frac{T_{\sigma}^{\#}\Phi}{\Phi}(x) \Phi(x) dx\\
& \geq \frac{1}{4}\sqrt{t}p,
\end{align*}
which contradicts \eqref{eq:contradiction-assumption}!
\end{proof}

To apply this lemma, we define for any $\delta\in \mathbb{R}$ and $v\in \mathbb{R}^W$ the map
$T_{\delta,v}:\mathbb{R}^{W^2}\to \mathbb{R}^{W^2}$ by 
$$T_{\delta, v}\left(B\right) = B + \delta \ket{Bv}\bra{v}.$$
$T_{\delta,v}$ has the following basic properties. 
\begin{proposition}(Properties of $T_{\delta,v}$)\label{prop:properties-of-T}
For any $\delta\in \mathbb{R}$ and unit vector $v\in \mathbb{R}^W$, 
$T_{\delta,v}$ is a linear map such that 
\begin{enumerate}
\item for each unit vector $w\in \mathbb{R}^W$ $T_{\delta,v}$ preserves the linear subspace
$$\left\{B\in \mathbb{R}^{W^2}\ :\ Bv = \left\|Bv\right\|w\right\},$$
\item $\det\left(T_{\delta,v}\right) = \left(1 + \delta\right)^{W},$ and thus $T_{\delta,v}$ is a smooth diffeomorphism 
if $\delta \neq -1$. 
\end{enumerate}
\end{proposition}
    
\begin{proof}
$\left(1\right)$ is trivial. 
For $\left(2\right)$, note that by rotating we can assume $v=e_1$. But, then $T_{\delta,v}$ 
simply multiplies the first column of $B$ by $(1+\delta)$, implying the claim.
\end{proof}

    To apply Lemma \ref{lem:MW} we need to compute the distortion of the law of $\left(B_j\,\middle|\, \mathcal{F}\right)$ under $T_{\delta,v}$, 
    i.e. estimate the quantity $\frac{T_+^{\#}\Phi}{\Phi}(X) \frac{T_-^{\#}\Phi}{\Phi}(X)$ 
    in our setting.
    By Proposition \ref{prop:schenker-product}, see \eqref{eq:joint-density},
    the density of $\left(B_j \,\middle|\, \left(D_{i}\right)_{i\in \left[1,N\right]}\right)$ is given by
    \begin{equation}\label{eq:conditional-B-density}
    F_j\left(B\right) := \frac{1}{Z}\phi_{1,j+1}\left(D_{j+1} + E + B^{\ast}D_j^{-1}B\right)\phi_{2,j}\left(B\right),
    \end{equation}
    where $\phi_{1,j+1}$ and $\phi_{2,j}$ are the unconditional densities 
    of $A_{j+1}$ and $B_j$, and we have used $Z$ as a normalization constant whose value
    plays no role in the proof. 

    \begin{lemma}\label{lem:distortion-general}
        Let $M>0$ and $H$ be as in Theorem ~\ref{thm:general-main}. Then, for any $j\in \left[1,N-1\right]$, $|\delta|\leq 2$, 
        $E\in \mathbb{R}$, $B\in \mathbb{R}^{W^2}$ and unit vector $v\in \mathbb{R}^W$, we have
        \begin{align*}
          \Big |\log F_j\left(T_{-\delta,v}\left(B\right)\right) + \log F_j\left(T_{\delta,v}\left(B\right)\right)  
          & -  2\log  F_j\left(B\right)\Big|\\
          &  \lsim \delta^2 W \left(\left\|Bv\right\|^2 + \left\|f_{j+1}(A_{j+1})v\right\|^2 + \left\|B^{\ast}D_{j}^{-1}Bv\right\|^2\right).
        \end{align*}
    \end{lemma}
    
    \begin{proof}
        We use the form of $F_j$ given in \eqref{eq:conditional-B-density}, 
        and first estimate the 
        contribution of $\phi_{2,j}$. If $X$ is a random variable with density $\phi(x)$,
        and $\sqrt{W}X$ is $M$-regular, then by \eqref{eq:general-distribution-condition}, we have 
        \begin{equation}\label{eq:bounded-second-derivative}
        \left|\log\frac{\phi(x+y)\phi(x-y)}{\phi(x)^2}\right| \lsim W\left|y\right|^2,
        \end{equation} 
        for any $x,y\in \mathbb{R}$. Thus, the distribution of each entry of $B_j$ satisfies \eqref{eq:bounded-second-derivative}. Since the entries of $B_j$ are independent this implies
        \begin{equation}\label{eq:second-order-cancellation}
        \left| \log\frac{\phi_{2,j}\left(B+P\right)\phi_{2,j}\left(B-P\right)}{\phi_{2,j}\left(B\right)^2} \right|\lsim W\left\|P\right\|_{F}^2,
        \end{equation}
        for any $P,B\in \mathbb{R}^{W\times W}$.
        Taking $P = \delta \ket{Bv}\bra{v}$ gives 
              $$\big|\log \phi_{2,j}\left(T_{-\delta,v}\left(B\right)\right) + \log \phi_{2,j}\left(T_{\delta,v}\left(B\right)\right) - 2\log \phi_{2,j}\left(B\right)\big|\lsim \delta^2 W\left\|Bv\right\|^2.$$
        
        To estimate the change in $\phi_{1,j+1}$, we define for any $B\in \mathbb{R}^{W\times W}$ 
        $$A(B) := D_{j+1} + E + B^{\ast}D_j^{-1}B.$$
        By the same argument as for $\phi_{2,j}$, 
        \eqref{eq:second-order-cancellation} holds for $\phi_{1,j+1}$ when $B$ and $P$ are symmetric. Thus
        \begin{align*}
            \Big |\log \phi_{1,j+1}\left(A\left(T_{-\delta,v}\left(B\right)\right)\right) + \log \phi_{1,j+1}& \left(A\left(T_{\delta,v}\left(B\right)\right)\right) \\
            & - 2\log \phi_{1,j+1}\left(A\left(B\right) + \delta^2\langle B^{\ast}D_{j}^{-1}Bv,v\rangle \ket{v}\bra{v}\right)\Big|\\
            & \lsim \delta^2W \left\|\ket{v}\bra{Bv}D_{j}^{-1}B + B^{\ast}D_{j}^{-1}\ket{Bv}\bra{v}\right\|_{F}^2\\
            & \lsim \delta^2 W\left\| B^{\ast}D_j^{-1}Bv\right\|^2.
        \end{align*}
        To finish we need to estimate
        $$\log\phi_{1,j+1}\left(A(B) + \delta^2\langle B^{\ast}D_{j}^{-1}Bv,v\rangle \ket{v}\bra{v}\right)- \log \phi_{1,j+1}\left(A(B)\right).$$
        For this we Taylor expand $\log \phi_{1,j+1}$ around $A\left(B\right)$.
        Recalling the definition of $f_{j+1}$ in \eqref{def:f-function}
        and that the entries of $\sqrt{W}A_{j+1}$ are $M$-regular and independent up to symmetry 
        we have
        $$\left|\log \phi_{1,j+1}(A+P) - \log \phi_{1,j+1}(A) - W\Tr\left(f_{j+1}(A)P\right)\right| \lsim W\|P\|_{F}^2,$$
        for any symmetric $A,P\in \mathbb{R}^{W\times W}_{Sym}$.
        Taking $A= A(B)$ and $P = \delta^2\langle B^{\ast}D_{j}^{-1}Bv,v\rangle \ket{v}\bra{v}$ gives
        \begin{align*}
            | \log \phi_{1,j+1}\big(A\left(B\right) & + \delta^2\langle B^{\ast}D_j^{-1}Bv,v\rangle\ket{v}\bra{v}\big)  
             -  \log \phi_{1,j+1}\left(A\left(B\right)\right) |\\ 
            & \lsim W\delta^2\left|\langle B^{\ast}D_{j}^{-1}Bv,v\rangle\right|\big|\Tr\left(f_{j+1}\left(A\left(B\right)\right)\ket{v}\bra{v}\right)\big| + W \delta^4\left|\langle B^{\ast}D_{j}^{-1}Bv,v\rangle\right|^2\\
            & \lsim W\delta^2 \left(\left\|f_{j+1}(A\left(B\right))v\right\|^2 + \left\|B^{\ast}D_{j}^{-1}Bv\right\|^2\right).
            \end{align*}
        To pass to the third line we used that $\left|\Tr\left(A\ket{v}\bra{v}\right)\right| \leq \left\|Av\right\|,$ and $|\delta|\leq 2$.
        Combining the above estimates gives
        \begin{align*}
        \big |\log \phi_{1,j+1}(A(T_{-\delta,v}\left(B\right))) +&  \log \phi_{1,j+1}(A(T_{\delta,v}\left(B\right)))  - 2\log \phi_{1,j+1}(A\left(B\right))\big|\\
        & \lsim \delta^2 W\left(\left\| B^{\ast}D_j^{-1}Bv\right\|^2 + \left\|Bv\right\|^2 + \left\|f_{j+1}(A\left(B\right))v\right\|^2\right).
        \end{align*}
        The lemma follows since, by definition,  $A\left(B\right) = A_{j+1}$. 
        \end{proof}
        
        Lemma \ref{lem:single-step-variance-general} now follows for $W\neq 4$
        by applying Lemma \ref{lem:MW} with $\delta = 2W^{-1/2}$,
        $X = \left(B_j\ | \mathcal{F}\right)\in \mathbb{R}^{W^2}$, 
        $g(B) = \log \left\|Bv_j\right\|,$ and
        $T_{\pm} = T_{\pm \delta, v_j}$. When $W=4$, we do the same but with $\delta = 3W^{-1/2}$
        so that $\text{det}\left(T_{-\delta, v_j}\right)\neq 0$. Indeed, since 
        $F_j$ is the density of $\left(B_j \, \middle|\, \left(D_i\right)_{i\in \left[1,N\right]}\right)$,
        the density of $\left(B_j \,\middle|\, \mathcal{F}\right)$, which is the same as  
        $\left(B_j\,\middle|\, D_j, D_{j+1}, v_j, v_{j-1}\right)$, is given by 
        $$\Phi(B) := \frac{1}{Z}\left\|Bv_j\right\|^{W-1}F_j(B),$$
        restricted to the $W^2-W+1$ dimensional space 
        \[\left\{B\in \mathbb{R}^{W\times W}\ : Bv_j = \frac{\left\|Bv_j\right\|}{\left\|D_jv_{j-1}\right\|}D_{j}v_{j-1}\right\}.\]
        By Proposition \ref{prop:properties-of-T}, $T_{\pm}$ are smooth diffeomorphisms preserving that subspace, 
        and by Lemma \ref{lem:distortion-general}, we have
        \begin{align*}
        \frac{T_+^{\#}\Phi}{\Phi}(B) \frac{T_-^{\#}\Phi}{\Phi}(B)
        & = \left(1+\frac{2}{\sqrt{W}}\right)^{2W-1} \left(1-\frac{2}{\sqrt{W}}\right)^{2W-1} \frac{F_j(T_{\delta,v_j}\left(B\right))F_{j}\left(T_{-\delta,v_j}\left(B\right)\right)}{F_j\left(B\right)^2}\\
        & \gsim e^{-C\left(\left\|Bv_j\right\|^2 + \left\|f_{j+1}(A_{j+1})v_j\right\|^2 + \left\|B^{\ast}D_{j}^{-1}Bv_j\right\|^2\right)}.
        \end{align*}
        Thus, Lemma \ref{lem:MW} implies
        $$\sup_{a\in \mathbb{R}}\mathbb{P}\left(\left| \alpha_j -a\right| \leq \frac{1}{\sqrt{W}}  \,\middle| \, \mathcal{F}\right) 
        \leq 1 - e^{-Ct}\mathbb{P}\left(\left\|B_jv_j\right\|, \left\|f_{j+1}(A_{j+1})v_j\right\|, \left\|B_j^{\ast}D_j^{-1}B_jv_j\right\| \leq t \,\middle| \,\mathcal{F}\right),$$
        for any $t\geq 1$. Note we can remove the $c$ in Lemma \ref{lem:MW} 
        by increasing $C$ and using $t\geq 1$.

\section{Proof of Theorem \ref{thm:single-energy-decay}}

Fix a unit vector $w\in \mathbb{R}^{W}$ and energy $E\in [-E_0,E_0]$. 
Recall that, by Proposition \ref{prop:schenker-product} and the definition of 
the $\alpha_j$, 
\begin{equation}
  \log \left\|G\left(1,N\right)w\right\| = \sum_{i=1}^{N} \alpha_i,
\end{equation}
and
the $\alpha_j$ are all independent random variables conditioned 
on $\mathcal{F}= \sigma\left(\left(D_i,v_i\right)_{i\in [1,N]}\right)$. Hence, 
\begin{equation}\label{eq:product-form}
\begin{aligned}
\left(\mathbb{E}\left\|G(1,N)w\right\|^q\right)^2 
& =\left(\mathbb{E} e^{q\sum_{i=1}^{N}\alpha_i}\right)^2\\ 
& \leq \mathbb{E}\left(\mathbb{E} \left[e^{q\sum_{i=1}^{N}\alpha_i} \,\middle| \, \mathcal{F}\right]\right)^2\\
& = \mathbb{E}\left(\prod_{i=1}^{N}\mathbb{E}\left[e^{q\alpha_i} \,\middle| \, \mathcal{F}\right]^2\right). 
\end{aligned}
\end{equation}
To estimate the RHS above, 
we use the following elementary lemma, whose proof is given in the appendix.
It is a simpler version of Proposition 3 in \cite{Schenker-2009}.
\begin{lemma}\label{lem:elementary-jensens}
    For any $p,\delta\in [0,1]$
    and random variable $X$ satisfying
    \begin{equation}\label{eq:elementary-jensens-assum}
    \sup_{a\in \mathbb{R}}\mathbb{P}\left(\left|X-a\right|\leq \delta\right)\leq 1-p,
    \end{equation}
    and $\mathbb{E}e^{X}<\infty$ we have
    $$\left(\mathbb{E} e^{X}\right)^2 \leq e^{-cp\delta^2}\mathbb{E}e^{2X}.$$
\end{lemma}

To apply this to \eqref{eq:product-form}, we let $t_0\geq 1$ 
be a constant to be chosen, and define the $\mathcal{F}$-measurable random variables
$$p_j := \mathbb{P}\left(\left\|B_jv_j\right\|, \left\|B_j^{\ast}D_{j}^{-1}B_jv_j\right\|, \left\|f_{j+1}\left(A_{j+1}\right)v_j\right\|\leq t_0 \,\middle| \, \mathcal{F}\right),$$
for all $j\in [1,N-1]$ so that Lemma \ref{lem:single-step-variance-general} implies
$$\sup_{a\in \mathbb{R}}\mathbb{P}\left(\left|\alpha_j-a\right|\leq W^{-1/2}\, \middle|\, \mathcal{F}\right)
\leq 1-e^{-Ct_0}p_j.$$
Hence applying Lemma \ref{lem:elementary-jensens} 
to each factor in \eqref{eq:product-form} with $X = q\alpha_j$ implies that
for any $q\in (0,1/5)$ 
\begin{equation}\label{eq:almost-done-version-2}
\begin{aligned}
\left(\mathbb{E}\left\|G(1,N)w\right\|^{q}\right)^{2} 
& \leq \mathbb{E}\left(e^{-cq^2W^{-1}\sum_{i=1}^{N-1}p_i}\mathbb{E}\left[e^{2q\sum_{i=1}^{N}\alpha_i}  \,\middle| \, \mathcal{F}\right]\right)\\
& \leq \left(\mathbb{E}e^{-cq^2W^{-1}\sum_{i=1}^{N-1}p_i}\right)^{1/2}\left(\mathbb{E}e^{4q\sum_{i=1}^{N}\alpha_i}\right)^{1/2}\\
& = \left(\mathbb{E}e^{-cq^2W^{-1}\sum_{i=1}^{N-1}p_i}\right)^{1/2} \left(\mathbb{E}\left\|G(1,N)\right\|^{4q}\right)^{1/2}\\
& \leq W^{Cq}\left(\mathbb{E}e^{-cq^2W^{-1}\sum_{i=1}^{N-1}p_i}\right)^{1/2},
\end{aligned}
\end{equation}
where $C,c$ may depend on $t_0$. 
We applied Cauchy-Schwarz and Jensen's inequality to pass 
to the second line and the Wegner estimate 
(Proposition \ref{prop:wegner}) to pass to the last line, 
using that $4q\leq 4/5$. 

Thus, the theorem follows if for some $t_0\geq 1$, depending only on $M$ and $E_0$, we can show
\begin{equation}\label{eq:whats-left}
\mathbb{P}\left(\sum_{i=1}^{N-1}p_i\geq cN\right) \geq 1-e^{-cN}.
\end{equation}
 To do this we define for each $j\in [1,N-1]$
$$X_j :=\max \left(\left\|B_jv_j\right\|, \left\|f_{j+1}\left(A_{j+1}\right)v_j\right\|, \left\|B_j^{\ast}D_j^{-1}B_jv_j\right\|\right),$$
so that Lemma \ref{lem:single-step-variance-general} implies
\begin{equation}\label{eq:lower-bound-sum-of-ps}
\sum_{i=1}^{N-1}p_i =  \sum_{i=1}^{N-1} \mathbb{P}\left(X_i\leq t_0\, \middle|\,\mathcal{F}\right) = \mathbb{E}\left[\sum_{i=1}^{N-1}1_{X_i\leq t_0}\,\middle|\, \mathcal{F}\right].
\end{equation}
To lower bound the RHS  we need to bound each  
quantity in $X_j$, uniformly in $W$. 

First we estimate $\left\|B_j^{\ast}D_j^{-1}B_j v_j\right\|$. 
Note $\left\|D_{j}^{-1}\right\|$ is typically order $W$, 
and so $\left\|B_jD_{j}^{-1}B_jv_j\right\|\lsim 1$
can only hold due to correlations between 
the $v_j, D_j$ and $B_j$ which force $v_j$ away from 
the most expanding directions of $B_{j}^{\ast}D_{j}^{-1}B_j$. 
This is captured by the following claim. 
\begin{claim}\label{claim:either-or}
    If $j\in [1,N-1]$ and $\left\|A_j\right\|_{op}, \left\|B_j\right\|_{op}\leq 10$, then 
    either  
    $$\left\|B_j^{\ast}D_{j}^{-1}B_jv_j\right\|\leq 100,$$
    or 
    $$\left\|B_{j-1}^{\ast}D_{j-1}^{-1}B_{j-1}v_{j-1}\right\| \leq 100 + \left|E\right|.$$
\end{claim}
\begin{proof}
    If $\left\|B_j^{\ast}D_{j}^{-1}B_jv_j\right\|\leq 100$ we are done, so suppose not. 
    In this case $v_{j}$ must be expanded under $D_j^{-1}B_j$. 
    Indeed, since $\left\|B_j\right\|_{op}\leq 10$ we must have
    $$\left\|D_j^{-1}B_j v_{j}\right\| \geq 10.$$
    But then by the definition of $D_j$ and $v_{j-1}$ we have 
    \begin{align*}
    \left\|B_{j-1}^{\ast}D_{j-1}^{-1}B_{j-1}v_{j-1}\right\|
    & = \left\|A_jv_{j-1} - Ev_{j-1} -D_{j}v_{j-1}\right\|\\
    & \leq 10 + \left|E\right| + \left\|D_{j}\frac{D_{j}^{-1}B_jv_{j}}{\left\|D_j^{-1}B_jv_{j}\right\|}\right\|\\
    & \leq 11 + \left|E\right|
\end{align*}
which implies the claim.
\end{proof}

To estimate $\|f_{j+1}(A_{j+1})v_j\|$, $\|B_jv_j\|$ and $\left\|A_j\right\|$ 
we use the following version of the Bai-Yin theorem that follows directly from 
Theorem 5.9 in
\cite{bai-silverstein}. For a simpler proof of a 
slightly weaker but sufficient estimate see \cite{Latala}.

\begin{proposition}\label{prop:bai-yin}
Let $\lambda,\epsilon>0$. If $A\in \mathbb{R}^{n\times n}$ is a random matrix with independent entries, 
or is symmetric with independent entries above the diagonal and 
$\mathbb{E}A_{ik} = 0$, $\mathbb{E}\left|A_{ik}\right|^2\leq 1$ and $\mathbb{E}\left|A_{ik}\right|^4 \leq \lambda$ 
for all $x,y\in \left[1,n\right]$ then 
$$\mathbb{P}\left(\left\|A\right\|_{op}\geq (2+\epsilon)\sqrt{n}\right)\leq \epsilon,$$
for $n$ sufficiently large, depending on $\lambda$ and $\epsilon$. 
\end{proposition}
Note that we can actually apply this proposition to 
estimate $\|f_{j+1}(A_{j+1})\|$. 
Indeed, if we let $\phi_{ik}:\mathbb{R}\to \mathbb{R}_{>0}$
be the density of the random variable $\left(A_{j+1}\right)_{ik}$, then 
one can check that
\begin{equation}\label{eq:f_j-in-coordinates}
\left(f_{j+1}(A_{j+1})\right)_{ik}
= 
\begin{cases}
\frac{1}{W}\frac{d}{dx}\left(\log \phi_{ik}\right)\left(\left(A_{j+1}\right)_{ik}\right) & \text{if } i=k,\\
\frac{1}{2W}\frac{d}{dx}\left(\log \phi_{ik}\right)\left(\left(A_{j+1}\right)_{ik}\right) & \text{if } i\neq k.
\end{cases} 
\end{equation}
Hence, $f_{j+1}(A_{j+1})$ is a symmetric matrix with independent entries on and above the diagonal. 
To estimate the moments of the entries we note the following elementary calculation. 
\begin{claim}
If $X$ is an $M$-regular random variable with density $\phi:\mathbb{R}\to\mathbb{R}_{>0}$
then we have $\mathbb{E}\frac{d}{dx}\left(\log \phi\right)(X)  = 0$, $\mathbb{E}\left(\frac{d}{dx}\left(\log \phi\right)(X)\right)^2 \leq M$, 
and $\mathbb{E}\left(\frac{d}{dx}\left(\log \phi\right)(X)\right)^4 \leq 3M^2$. 
\end{claim} 
\begin{proof}
    Writing $h(x) = (\log \phi)(x)$, the fundamental theorem of calculus implies
    $$\mathbb{E}\frac{d}{dx}\left(\log \phi \right)(X) = \int_{\mathbb{R}} h'(x)e^{-h(x)}dx =0.$$
    Furthermore, using integration by parts we can estimate
    \begin{align*}
    \mathbb{E}\left(\frac{d}{dx}\left(\log \phi\right)(X)\right)^2
    & = \int_{\mathbb{R}} h'(x)^2e^{-h(x)}dx = \int_{\mathbb{R}}h{''}(x)e^{-h(x)}dx\leq M,
    \end{align*}
    and 
    \begin{align*}
    \mathbb{E}\left(\frac{d}{dx}\left(\log \phi\right)(X)\right)^4 
     = \int_{\mathbb{R}} h'(x)^4e^{-h(x)}dx
     = 3\int_{\mathbb{R}} h'(x)^2h''(x)e^{-h(x)}dx\leq 3M^2.
    \end{align*}
    Note that the boundary terms vanished since $X$ is $M$-regular.
\end{proof}
By scaling, this implies each entry of $f_{j+1}(A_{j+1})$ is mean $0$, 
has second moment at most $MW^{-1}$ and fourth moment at most $3M^2W^{-2}$. 
Hence, 
for $W$ sufficiently large, depending on $M$, Proposition \ref{prop:bai-yin} 
and a union bound imply
\begin{equation}\label{eq:final-probability-estimate}
\mathbb{P}\left(\sup_{i=j,j-1}\left\|A_i\right\|, \left\|B_i\right\|\leq 10, \sup_{i=j,j+1}\left\|f_{i}(A_{i})\right\| \leq 10(1+M)\right)\geq 1/2,
\end{equation}
for $j\in [2,N-1]$. 

In the above event, Claim \ref{claim:either-or} implies either $X_j\leq 100(1+M + \left|E\right|)$ or $X_{j-1}\leq 100(1+M + \left|E\right|)$, 
and so taking $t_0 =100(1+M + \left|E\right|)$, using the independence of the $A_i$ and $B_i$ and assuming 
$N>3$ (if $N\leq 3$ the theorem is trivial) gives
\begin{align*}
\mathbb{P}\left(\sum_{i = 1}^{N-1} 1_{X_i\leq t_0} \geq \frac{N}{10} \right)
& \geq 1-2e^{-c N}.
\end{align*}
The law of total probability and \eqref{eq:lower-bound-sum-of-ps} 
then imply the $\mathcal{F}$-measurable event
$$\sum_{i=1}^{N-1} p_i = \mathbb{E}\left[\sum_{i=1}^{N-1}1_{X_i\leq t_0}\, \middle| \, \mathcal{F}\right]\gsim N\mathbb{P}\left(\sum_{i=1}^{N-1}1_{X_i\leq t_0}\geq \frac{N}{10}\, \middle|\, \mathcal{F}\right) \gsim N,$$
holds with probability at least $1-2e^{-c N}$, which combined with \eqref{eq:almost-done-version-2} 
proves the theorem for $N>0$ and $W$ sufficiently large, depending on $M$.
For bounded $W>0$, the same argument works with the constants in \eqref{eq:final-probability-estimate}
and Claim \ref{claim:either-or} adjusted accordingly.  

\section{Appendix}

\subsection{Proof of Lemma \ref{eq:elementary-jensens-assum}}
Without loss of generality we can assume $\mathbb{E}e^{X} = 1$. 
Since $\delta\in [0,1]$, \eqref{eq:elementary-jensens-assum} 
implies $\mathbb{P}(|e^{X} - 1|\geq c\delta)\geq p$ so we have
\begin{align*}
    \mathbb{E}e^{2X}  
    & = 1 + \mathbb{E}\left(e^{X} - 1\right)^2\\
    & \geq 1 + cp\delta^2\\
    & \geq e^{cp\delta^2}.
\end{align*}
using that $p,\delta\in [0,1]$ to pass to the third line. The claim follows.

\subsection{Proof of Wegner Estimate}

We follow the arguments of
Lemma 1 in \cite{Bourgain-2013} and Proposition 3.4 of \cite{Schlag}.  
First we prove the following claim. It says 
a random diagonal matrix plus deterministic one has random eigenvalues
in a quantitative way.

\begin{claim}
    \label{claim:general-wegner-claim}
    Let $M>0$. If $D\in \mathbb{R}^{n\times n}$ is a random diagonal matrix
    with independent $M$-regular entries on the diagonal, 
     $Q\in \mathbb{R}^{n\times n}$ is any symmetric matrix, and $\epsilon>0$ then 
    $$\mathbb{P}\left(\left|\sigma(D+Q)\cap\left[-\epsilon,\epsilon\right]\right| \geq 1 \right) \lsim \epsilon n.$$
    \end{claim}

We will use the following version of the Schur complement formula. 
\begin{proposition}(Schur Complement Formula)\label{prop:schur-complement}
If 
$$A = 
\begin{pmatrix}
A_{11} & A_{12}\\
A_{21} & A_{22}
\end{pmatrix}$$
is a real or complex block matrix and $A$ and $A_{22}$ are invertible, then
$$\left(A^{-1}\right)_{11} = \left(A_{11}-A_{12}A_{22}^{-1}A_{21}\right)^{-1}.$$
\end{proposition}
\begin{proof}[Proof of Proposition]
Viewing $A^{-1}$ in the same block form as $A$ gives the equations
\begin{align*}
Id & = A_{11}\left(A^{-1}\right)_{11} + A_{12}\left(A^{-1}\right)_{21}\\
0 & = A_{21}\left(A^{-1}\right)_{11} + A_{22}\left(A^{-1}\right)_{21}.
\end{align*}
Solving the second equation for $\left(A^{-1}\right)_{21}$ 
and substituting into the first gives the result. 
\end{proof}
Now we prove the claim. 
    \begin{proof}[Proof of Claim]
    The idea is to move the spectrum of $D+Q$ by varying the entries of $D$.  
    If $\lambda_1,...,\lambda_n$ are the eigenvalues of $D+Q$, then 
    \begin{align*}
    \mathbb{E}\left|\sigma(D+Q)\cap[-\epsilon,\epsilon]\right| 
    & \lsim \epsilon \mathbb{E}\sum_{i=1}^{n} \frac{\epsilon}{\lambda_i^2 + \epsilon^2}\\
    & \lsim \epsilon \text{Im} \Tr\left((D+Q-i\epsilon)^{-1}\right)\\
    & = \epsilon\mathbb{E} \sum_{x\in [1,n]}\text{Im}\left(D+Q-i\epsilon\right)^{-1}\left(x,x\right)
    \end{align*}
    To estimate each term in the sum, we apply the Schur complement formula. 
    Indeed, for any $x\in [1,n]$ the Schur complement formula 
    applied to $D+Q - i\epsilon$ with $A_{11} = \left(D+Q-i\epsilon\right)_{xx}$ gives
    $$\left(D+Q-i\epsilon\right)^{-1}\left(x,x\right) = \left(D_{xx} - z_0\right)^{-1},$$
    where $z_0\in \mathbb{C}$ is a complex number with $\text{Im}z_0>0$, 
    depending on all entries of $D$ except $D_{xx}$. Hence if $\phi$ 
    is the density of $D_{xx}$ we have 
    \begin{align*}
    \mathbb{E}\text{Im}\left(D+Q-z_0\right)^{-1}\left(x,x\right) \lsim\sup_{z\in \mathbb{C}, \text{Im}z>0}\int_{\mathbb{R}}\frac{\text{Im}\ z}{\left(\text{Im}\ z\right)^2 + \left(\text{Re}\ z-x\right)^2}\phi(x) dx\leq \left\|\phi\right\|_{L^{\infty}}\lsim 1,
    \end{align*}
    since $X$ being $M$-regular implies $\phi$ is bounded (see \eqref{eq:bounded-second-derivative} for instance).
    Hence 
    $$\mathbb{P}\left(\left|\sigma(D+Q)\cap\left[-\epsilon,\epsilon\right]\right|\geq 1\right)
    \leq \mathbb{E}\left|\sigma(D+Q)\cap[-\epsilon,\epsilon]\right| \lsim \epsilon n.$$
    \end{proof}

Now the Wegner estimate follows easily.  
Applying the Schur complement formula with 
$A = H-E$ and $A_{11}$ 
being the $2W\times 2W$ matrix given by 
$A_{11} = H_{\{i,j\}}- E$, 
gives
$$
\begin{pmatrix}
G(i,i) & G(i,j)\\
G(j,i) & G(j,j)
\end{pmatrix}
= 
\left(\begin{pmatrix}
    H_{ii} & 0\\
    0 & H_{jj}
    \end{pmatrix}
    + Q\right)^{-1}
    := \left(\tilde{D} + Q\right)^{-1}.
$$
where $Q$ is a $2W\times 2W$ symmetric matrix
which is \textit{independent of $\left(H_{ii}, H_{jj}\right)$}.
Note that Claim \ref{claim:general-wegner-claim} implied the 
necessary matrices were invertible almost surely.  
Finally, the entries of $\sqrt{W}H_{ii}$ and $\sqrt{W}H_{jj}$ are 
$M$-regular, so Claim \ref{claim:general-wegner-claim} implies 
\begin{align*}
\mathbb{P}\left(\left\|G(i,j)\right\|>\lambda\right)
& \leq \mathbb{P}\left(\left|\sigma\left(\tilde{D} + Q\right)\cap\left[-\lambda^{-1},\lambda^{-1}\right]\right|\geq 1\right)\\
& = \mathbb{P}\left(\left|\sigma\left(\sqrt{W}\tilde{D} + \sqrt{W}Q\right)\cap\left[-\sqrt{W}\lambda^{-1},\sqrt{W}\lambda^{-1}\right]\right|\geq 1\right)\\
& \lsim W^{3/2}\lambda^{-1},
\end{align*}
which proves the claim.

\printbibliography

\end{document}